\documentclass[15pt]{article}

\usepackage{amsmath,amsfonts,amssymb,amsthm,amscd,mathrsfs}
\usepackage{enumitem}

\usepackage{yhmath}
\usepackage{graphicx} % for \scalebox macro

\newcounter{stepctr}
{\end{list}}

\newtheorem{thm}{Theorem}[section]
\newtheorem{prop}[thm]{Proposition}
\newtheorem{cor}[thm]{Corollary}

\theoremstyle{definition}
\newtheorem{dfn}[thm]{Definition}

\newtheorem{rema}[thm]{Remark}

\newtheorem{prob*}{Open problem}
\newcommand{\demo}{\begin{proof}}

\newcommand{\A}{\mathcal{A}}

\newcommand{\R}{\ensuremath{\mathcal{R}}}

\newcommand{\N}{\mathbb{N}}

\newcommand{\C}{\mathbb{C}}

\def\ll^2{{\mathcal L}(\ell^2(\N))}

\def\f^0x{{\mathcal F^0}(X) }

\title
{\bf  Almost  invertible operators}

\author{Zakariae Aznay, Abdelmalek  Ouahab,  Hassan  Zariouh}

\date{}
\begin{document}

\maketitle \thispagestyle{empty}

\begin{abstract} \baselineskip=10pt
\noindent We prove that  a bounded  linear operator $T$ is  a direct sum of an invertible operator and   an operator with at most countable spectrum  iff $0\notin\mbox{acc}^{\omega_{1}}\,\sigma(T),$ where $\omega_{1}$  is  the smallest uncountable ordinal and  $\mbox{acc}^{\omega_{1}}\,\sigma(T)$ is the $\omega_{1}$-th Cantor-Bendixson derivative  of $\sigma(T).$
\end{abstract}

 \baselineskip=15pt
  \baselineskip=15pt
\section{Introduction}
  Let  $L(X)$ be  the set of all bounded linear operators acting on   a complex Banach space $X$  and denote, respectively,  the spectrum, the Browder spectrum and the Drazin spectrum on $L(X)$  by   $\mathcal{\sigma},$ $\sigma_{b}$ and $\sigma_{d},$ see \cite{aznay-ouahab-zariouh6} for more details. Let $T \in L(X).$ A subset $\tilde{\sigma}\subset\sigma(T)$ is called a spectral set (or isolated part) of $T$ if it  is clopen in $\sigma(T).$ A subspace $M$  is called $T$-invariant if $T(M) \subset M,$  in this case  the restriction of $T$ on $M$  is denoted by $T_{M}.$ Denote further by $(M,N)\in \mbox{Red}(T)$  when $M,N$ are closed $T$-invariant subspaces and $X=M\oplus N,$ and denote by: 
$$
    \begin{array}{ll}
   \overline{A} &: \text{  the closure of a   complex subset } A\\
   A^C &: \text{ the complementary of a complex subset } A\\
   C(\lambda, \epsilon) &: \text{ the  circle  ball of radius } \epsilon  \text{ centered at } \lambda \\
   B(\lambda, \epsilon) &: \text{ the  open ball of radius } \epsilon  \text{ centered at } \lambda \\
   D(\lambda, \epsilon) &: \text{ the  closed  ball of radius } \epsilon  \text{ centered at } \lambda\\
    \end{array}
$$
For ordinal numbers $\alpha,$ the $\alpha$-th Cantor-Bendixson derivative of  $A\subset\C$ is defined by repeatedly applying the derived set operation using transfinite recursion as follows:
$$ \begin{cases}
       \mbox{acc}^{0}\,A=A\\
       \mbox{acc}^{\alpha}\,A=\mbox{acc}^{\alpha-1}\,(\mbox{acc}\,A) \text{ if  $\alpha$ is a successor   ordinal}\\
       \displaystyle \mbox{acc}^{\alpha}\,A=\bigcap_{\beta<\alpha}\mbox{acc}^{\alpha}\,A \text{ if  $\alpha$  is a limit ordinal}
       \end{cases}
$$
where $\mbox{acc}\,A$ is the set of all acumulation points of $A,$ and $\alpha-1$ denotes the  predecessor of  $\alpha$ if it is a successor.  The Cantor-Bendixson rank of the set  $A,$ denoted $\mbox{CBR}(A),$ is the smallest ordinal $\alpha$ satisfying   $\mbox{acc}^{\alpha}\,A=  \mbox{acc}^{\alpha+1}\,A.$  Recall that a separable  topological space $X$   is called Polish if it is   completely metrizable.    Let  $\sigma_{*}\in\{\sigma,\sigma_{b},\sigma_{d}\}.$ For every $T \in L(X),$ $\sigma_{*}(T)$ is Polish. This entails from \cite[Theorem 4.9]{kechrisbbook} that  $\mbox{CBR}(\sigma_{*}(T))<\omega_{1},$ where $\omega_{1}$  is  the second number class (i.e. the first uncountable ordinal).    It is clear that for  $\lambda:=\mbox{CBR}(\sigma_{*}(T)),$ $(\mbox{acc}^{\alpha}\,\sigma_{*}(T))_{\alpha\leq \lambda}$   is a strictly decreasing   family of compact subsets. Moreover, $$\displaystyle\mbox{acc}^{\beta}\,\sigma_{*}(T)\setminus\mbox{acc}^{\alpha}\,\sigma_{*}(T)= \bigsqcup_{\beta\leq\gamma<\alpha}\mbox{iso}\,(\mbox{acc}^{\gamma}\,\sigma_{*}(T)) \text{ for every ordinals } \alpha,\beta \text{ such that } \alpha>\beta,$$ where $\bigsqcup$ denotes the mutually disjoint union. Therefore 
$$\sigma_{*}(T)=\mbox{acc}^{\alpha}\,\sigma_{*}(T)\bigsqcup[\bigsqcup_{\beta<\alpha}\mbox{iso}\,(\mbox{acc}^{\beta}\,\sigma_{*}(T))] \text{ for all ordinals } \alpha.$$ As  $\mbox{acc}^{\alpha+1}\,\sigma(T)\subset\mbox{acc}^{\alpha}\,\sigma_{d}(T)\subset\mbox{acc}^{\alpha}\,\sigma_{b}(T)\subset\mbox{acc}^{\alpha}\,\sigma(T)$ for all ordinals  $\alpha,$  then  $\mbox{acc}^{\alpha}\,\sigma(T)=\mbox{acc}^{\alpha}\,\sigma_{b}(T)=\mbox{acc}^{\alpha}\,\sigma_{d}(T)$ for all  ordinals  $\alpha\geq \omega,$ where $\omega$ is the first infinite   ordinal. Thus  $\mbox{CBR}(\sigma(T))=\mbox{CBR}(\sigma_{b}(T))=\mbox{CBR}(\sigma_{d}(T))$  if $\mbox{CBR}(\sigma_{d}(T))\geq \omega,$ and if  $\mbox{CBR}(\sigma_{d}(T))< \omega,$  then   $$\mbox{CBR}(\sigma(T))-1\leq\mbox{CBR}(\sigma_{d}(T))\leq\mbox{CBR}(\sigma_{b}(T))\leq \mbox{CBR}(\sigma(T))\leq \mbox{CBR}(\sigma_{b}(T))+1.$$     
We   deduce  also that 
$$\sigma_{*}(T)=\mbox{acc}^{\alpha}\,\sigma(T)\bigsqcup[\bigsqcup_{\beta<\alpha}\mbox{iso}\,(\mbox{acc}^{\beta}\,\sigma_{*}(T))] \text{ if } \alpha\geq\omega.$$
   For $T \in L(X),$ it is proved in \cite{abad-zguiti,berkanimerom,koliha} that   $0\notin  \mbox{acc}\,\sigma_{*}(T)$ iff   $T=D\oplus Q,$  $D$ is an invertible operator  and $Q$ is an operator with $\sigma_{*}(Q)\subset\{0\}.$ Furthermore, we  showed   in \cite{aznay-ouahab-zariouh6} that $0\notin  \mbox{acc}^{2}\,\sigma(T)$ iff   $T=D\oplus Q,$  $D$ is invertible and $\mbox{acc}\,\sigma(Q)\subset\{0\}.$ In the present  work, we   generalize  these results by proving that   for each   successor   ordinal $\alpha,$ $0\notin  \mbox{acc}^{\alpha}\,\sigma_{*}(T)$ iff   $T=D\oplus Q,$  $D$ is invertible and $\mbox{acc}^{\alpha-1}\,\sigma_{*}(Q)\subset\{0\}.$  This  enables  us to characterize  the class of operators having  a direct decomposition  of an invertible operator  and  an operator with   at most countable  spectrum. More precisely, we prove that  $T$ is an  almost invertible operator iff $0\notin \displaystyle\mbox{acc}^{\omega_{1}}\,\sigma(T).$   Among other things, we show that if  $T$ has a topological uniform descent, then $T$ is Drazin invertible iff it  is an  almost invertible operator. 
\section{$g_{\sigma_{*}}^{\alpha}$-invertible operators}
Following    \cite{aznay-ouahab-zariouh6},  an operator $T\in L(X)$ is called zeroloid if $\mbox{acc}\,\sigma(T)\subset\{0\}.$ The following proposition  shows that  the product of commuting zeroloid operators is also  zeroloid. 
\begin{prop} Let  $T,S \in L(X)$  such that $TS=ST.$ The following assertions hold:\\
(i) If both $T$ and $S$ are   zeroloid, then $TS$ is zeroloid.\\
(ii)  If $0\in \mbox{acc}\,\sigma(TS),$ then $0\in \mbox{acc}\,\sigma(T)\cup \mbox{acc}\,\sigma(S).$
\end{prop}
\begin{proof}    Denote by $\A=\{T,S\}^{''}$ the second commutant of  $\{T,S\}$ (see \cite[Definition I.1.23]{mullerbook} for the definition). It is well known that $\A$  is a unital commutative Banach algebra, and if $\sigma_{\A}(T)$ is the spectrum   of $T$ in  $\A,$  then  $\sigma_{\A}(T)=\sigma(T).$ \\
(i)  Assume that $T,S$ are  non-zero.     As  $T$ and $S$ are zeroloid then $B(0,\epsilon)^{C}\cap\sigma(T)$ and  $B(0,\epsilon)^{C}\cap\sigma(S)$ are finite sets for every $\epsilon>0.$  We deduce,  from \cite[Theorem I.2.9]{mullerbook}, that  there exist   multiplicative functionals $\phi_{1},\dots,\phi_{n},\psi_{1},\dots,\psi_{m}$ on $\A$  (for the definition of a multiplicative functional, one can see \cite[Definition I.2.2]{mullerbook})    such that  $B(0,\mu)^{C}\cap[\sigma(T)\cup\sigma(S)]= \{\phi_{1}(T),\dots,\phi_{n}(T),\psi_{1}(S),\dots,\psi_{m}(S)\},$ where $\mu:=\frac{\epsilon}{\mbox{max}\{\|T\|,\|S\|\}}.$  Let $\phi$ be a multiplicative functional  such  that $|\phi(TS)|\geq \epsilon.$ Then  $\mbox{min}\{|\phi(T)|,|\phi(S)|\}\geq \mu,$ and thus   there exist $1\leq i\leq n$ and $1\leq j\leq m$ such that $\phi(TS)=\phi_{i}(T)\psi_{j}(S).$  Hence  $B(0,\epsilon)^{C}\cap\sigma(TS)$ is a finite set, and therefore $TS$ is zeroloid. \\
(ii)  If    $0\in \mbox{acc}\,\sigma(TS),$ then $B(0,\epsilon)\cap\left(\sigma(TS)\setminus\{0\}\right)\neq\emptyset$ for all   $\epsilon>0.$  So there exists a multiplicative functional $\phi_{\epsilon}$ on  $\A$  such that $0<|\phi_{\epsilon}(TS)|<\epsilon^{2}.$ Hence $|\phi_{\epsilon}(T)|<\epsilon$  or $|\phi_{\epsilon}(S)|<\epsilon.$ This  implies that  $B(0,\epsilon)\cap\left([\sigma(T)\cup\sigma(S)]\setminus\{0\}\right)\neq\emptyset,$  and then $0\in \mbox{acc}\,\sigma(T)\cup \mbox{acc}\,\sigma(S).$ 
\end{proof}
  The next definition describes the classes of operators we are going to study. 
\begin{dfn} Let  $T\in L(X)$ and let  $\alpha$ be an  ordinal.\\
 (1)   If $\alpha$ is a  successor ordinal,  we say that   $T$ is $g_{\sigma_{*}}^{\alpha}$-invertible if  there exists  $(M,N)\in \mbox{Red}(T)$   such that $T_{M}$  is invertible and $\mbox{acc}^{\alpha-1}\,\sigma_{*}(T_{N})\subset\{0\};$ where $\alpha-1$ is  the  predecessor of  $\alpha.$  \\
(2)    If $\alpha$ is a limit   ordinal,   we say that   $T$ is $g_{\sigma_{*}}^{\alpha}$-invertible if  there exists  a  successor ordinal $\beta<\alpha$  such that  $T$ is $g_{\sigma_{*}}^{\beta}$-invertible.\\
(3)  $T$ is called   $g_{\sigma_{*}}^{0}$-invertible if  $0\notin \sigma_{*}(T).$\\
 The class  of   $g_{\sigma_{*}}^{\alpha}$-invertible operators  is denoted by  $\mathcal{G}^{\alpha}_{\sigma_{*}}(X).$    
\end{dfn}
\begin{rema}(1) It is easily seen that   $\mathcal{G}_{\sigma}^{\alpha}(X)\subset\mathcal{G}_{\sigma_{b}}^{\alpha}(X)\subset\mathcal{G}_{\sigma_{d}}^{\alpha}(X)\subset\mathcal{G}_{\sigma}^{\alpha+1}(X)$ and   $\mathcal{G}^{\alpha}_{\sigma_{*}}(X)\subset\mathcal{G}^{\beta}_{\sigma_{*}}(X)$       for all ordinals $\alpha\leq\beta.$  \\
(2)   Let $\alpha$ be a limit   ordinal.  It is not difficult to see that       $T$ is $g_{\sigma_{*}}^{\alpha}$-invertible  iff   there exists  an ordinal $\beta<\alpha$  such that  $T$ is $g_{\sigma_{*}}^{\beta}$-invertible.  Thus $\displaystyle\mathcal{G}^{\alpha}_{\sigma_{*}}(X)=\bigcup_{\beta<\alpha}\mathcal{G}^{\beta}_{\sigma_{*}}(X).$
\end{rema}
 Recall \cite{koliha} that an operator $T$ is called   generalized Drazin invertible if there exists  $(M,N)\in \mbox{Red}(T)$ such that $T_{M}$ is invertible and  $\mbox{acc}\,\sigma(T_{N})\subset\{0\}.$ It was proved that $T$ is    generalized Drazin invertible  iff $0\notin \mbox{acc}\,\sigma(T).$ The following theorem gives a characterization of $g_{\sigma_{*}}^{\alpha}$-invertible operators in terms of $\mbox{acc}^{\alpha}\,\sigma_{*}(T).$  
 \begin{thm}\label{thmpgsigmaetoile}
Let  $T\in L(X)$ and  let  $\alpha$ be an ordinal. Then  $T-\mu I $  is $g_{\sigma_{*}}^{\alpha}$-invertible iff   $\mu \notin  \mbox{acc}^{\alpha}\,\sigma_{*}(T).$  
\end{thm}
\begin{proof}  Let us firstly prove that the result is true for all  successor  ordinals $\alpha.$  Let  $\mu \notin \mbox{acc}^{\alpha}\,\sigma_{*}(T).$  Since $\mbox{acc}^{\alpha}\,\sigma_{*}(T- z I)=\mbox{acc}^{\alpha}\,\sigma_{*}(T)-z$ for every complex number  $z,$ we may   assume that  $\mu=0.$    The case of $0 \notin \mbox{acc}\,\sigma(T)$  is clear.  Suppose now  that  $0 \in \mbox{acc}\,\sigma(T).$   Then $0\in \mbox{acc}\,(\Pi_{\sigma_{*}}^{\alpha}(T)),$ where $\Pi_{\sigma_{*}}^{\alpha}(T):=\sigma(T)\setminus \mbox{acc}^{\alpha-1}\,\sigma_{*}(T).$  As  $0\notin  \mbox{acc}^{\alpha}\,\sigma_{*}(T)$ then     $D(0,\epsilon)\cap [\mbox{acc}^{\alpha-1}\,\sigma_{*}(T)\setminus\{0\}]=\emptyset$ for some $\epsilon>0.$    Let $0<\mu<v\leq \epsilon.$   The set  $\Pi_{\sigma_{*}}^{\alpha}(T)$ is countable, since   $\lambda:=\mbox{CBR}(\sigma(T))<\omega_{1}$ and    $\Pi_{\sigma_{*}}^{\alpha}(T)\subset\sigma(T)\setminus\mbox{acc}^{\lambda}\,\sigma(T)=\displaystyle\bigsqcup_{\beta<\lambda}\mbox{iso}\,(\mbox{acc}^{\beta}\,\sigma(T)).$  Thus   there exists   $\mu\leq\epsilon_{(\mu,v)}\leq v$ such that     $C(0,\epsilon_{(\mu,v)})\cap\sigma(T)=\emptyset.$ Hence     $\sigma_{(\mu,v)}:= D(0,\epsilon_{(\mu,v)})  \cap\sigma(T)\subset \Pi_{\sigma_{*}}^{\alpha}(T)\cup\{0\}$ is a  countable spectral set of $\sigma(T).$   So    there exists $(M_{(\mu,v)}, N_{(\mu,v)})\in  \mbox{Red}(T)$  such that $\sigma(T_{M_{(\mu,v)}})=\sigma(T)\setminus\sigma_{(\mu,v)}$ and $\sigma(T_{N_{(\mu,v)}})=\sigma_{(\mu,v)}.$      Further, we have $\sigma_{(\mu,v)}\cap( \mbox{acc}^{\alpha-1}\,\sigma_{*}(T)\setminus\{0\})=\emptyset,$ and  thus     $\mbox{acc}^{\alpha-1}\,\sigma_{*}(T_{N_{(\mu,v)}})\subset \{0\}.$  This proves that  $T$ is $g_{\sigma_{*}}^{\alpha}$-invertible. 
Assume now that  $\alpha$ is a limit  ordinal.   If $0 \notin  \mbox{acc}^{\alpha}\,\sigma_{*}(T),$  then  there exists a successor   odinal  $\beta<\alpha$ such that $0 \notin  \mbox{acc}^{\beta}\,\sigma_{*}(T),$  which implies by the first case  that $T$ is $g_{\sigma_{*}}^{\beta}$-invertible, and thus $T$ is $g_{\sigma_{*}}^{\alpha}$-invertible.\\
Conversely, assume  that  $T$ is $g_{\sigma_{*}}^{\alpha}$-invertible for some ordinal $\alpha.$  Assume  that     $\alpha$ is a successor ordinal, then   there exists  $(M,N)\in \mbox{Red}(T)$   such that $T_{M}$  is invertible and $\mbox{acc}^{\alpha-1}\,\sigma_{*}(T_{N})\subset\{0\}.$ So there exists $\epsilon>0$ such that $B(0,\epsilon)\setminus\{0\}\subset(\sigma(T_{M}))^{C}\cap(\mbox{acc}^{\alpha-1}\,\sigma_{*}(T_{N}))^{C}\subset(\mbox{acc}^{\alpha-1}\,\sigma_{*}(T_{M}))^{C}\cap(\mbox{acc}^{\alpha-1}\,\sigma_{*}(T_{N}))^{C}=(\mbox{acc}^{\alpha-1}\,\sigma_{*}(T))^{C}.$ So  $0 \notin  \mbox{acc}^{\alpha}\,\sigma_{*}(T).$   The case of $\alpha$ is a limit ordinal is clear.  
\end{proof}
The following corollary is a  consequence of the previous theorem:
\begin{cor}   Let $T \in L(X)$ and let $\alpha$ be an ordinal.   The  following statements  hold:\\
(i)  $T$  is $g_{\sigma_{*}}^{\alpha}$-invertible iff  its dual adjoint  $T^{*}$  is $g_{\sigma_{*}}^{\alpha}$-invertible.\\
(ii) If $S$ is a bounded operator acts on a complex Banach space $Y,$  then  $S\oplus T$   is $g_{\sigma_{*}}^{\alpha}$-invertible iff both $S$ and $T$   are   $g_{\sigma_{*}}^{\alpha}$-invertible. \\
 (iii)   $T$  is $g_{\sigma_{*}}^{\alpha}$-invertible iff   $T^{m}$  is $g_{\sigma_{*}}^{\alpha}$-invertible for some  (equivalently for every) integer $m\geq 1.$
\end{cor}
Since $\mbox{CBR}(\sigma(T))<\omega_{1}$ and $\mbox{acc}^{\omega}\,\sigma(T)=\mbox{acc}^{\omega}\,\sigma_{b}(T)=\mbox{acc}^{\omega}\,\sigma_{d}(T)$  for all $T \in L(X),$ we deduce,  from Theorem \ref{thmpgsigmaetoile},  the following   corollary, which in particular shows  that the  $g_{\sigma_{*}}^{\alpha}$-invertibility is important only for countable ordinals.
\begin{cor} Let $T \in L(X).$ The following assertions hold:\\
 (i)   If  $\alpha\geq \omega_{1}$ is  an ordinal, then   $T$ is $g_{\sigma_{*}}^{\alpha}$-invertible iff $T$ is $g_{\sigma_{*}}^{\beta}$-invertible  for some $\beta<\omega_{1}.$\\
(ii)  If  $\alpha\geq\omega,$  then $T$ is $g_{\sigma_{*}}^{\alpha}$-invertible iff $T$ is $g_{\sigma_{**}}^{\alpha}$-invertible;   where  $\sigma_{**}\in\{\sigma, \sigma_{b},\sigma_{d}\}.$
\end{cor}

\begin{rema} Let  $T \in L(X)$ and   let  $$\mathcal{C}(T):=\{(T-zI)_{M}: z \in \C \text{ and } M \text{ is a closed  $T$-invariant subspace}\},$$  and denote the set of all subsets of $\mathbb{C}$ by $\mathcal{P}(\mathbb{C}).$  One can see from the proof  of  Theorem \ref{thmpgsigmaetoile}  that if  $\tilde{\sigma}: \mathcal{C}(T)  \longrightarrow \mathcal{P}(\mathbb{C})$   is a mapping that satisfies:\\
(i) $\tilde{\sigma}(T_{M})$ is closed and  $\tilde{\sigma}(T_{M})\subset\sigma (T_{M})$  for  all  closed $T$-invariant subspace   $M;$\\
(ii) $\tilde{\sigma}(T)=\tilde{\sigma}(T_{M})\cup\tilde{\sigma}(T_{N})$ for all $(M,N)\in \mbox{Red}(T);$\\
(iii) $\tilde{\sigma}(T-zI)=\tilde{\sigma}(T)-z$ for all complex numbers $z;$\\
(iv) $ \mbox{acc}^{\beta}\,\sigma(T)\subset \mbox{acc}^{\gamma}\,\tilde{\sigma}(T)$ for some ordinals $\beta,\gamma;$\\
then for all  successor  ordinal $\alpha,$    $z \notin  \mbox{acc}^{\alpha}\,\tilde{\sigma}(T)$ iff    there exists $(M,N)\in \mbox{Red}(T)$ such that $(T-zI)_{M}$ is invertible and $\mbox{acc}^{\alpha-1}\,\tilde{\sigma}((T-zI)_{N})\subset\{0\}.$ 
\end{rema}
\begin{dfn} Let $\alpha$ be an ordinal and let $T \in L(X)$ be a $g_{\sigma_{*}}^{\alpha}$-invertible. The  least ordinal $\beta$  such that    $0\notin\mbox{acc}^{\beta}\,\sigma_{*}(T)$ will be  called the degree of $g_{\sigma_{*}}$-invertibility of $T,$   denoted by $\beta:=\mbox{d}_{\sigma_{*}}(T).$  
\end{dfn}   
\begin{rema} (1)  If  $T \in L(X)$  is $g_{\sigma_{*}}^{\alpha}$-invertible   for some ordinal $\alpha$ and $0\in \sigma_{*}(T),$  then    the degree of $g_{\sigma_{*}}$-invertibility of $T$ is  the      ordinal  $\beta$  such that    $0\in\mbox{iso}\,\mbox{acc}^{\beta}\,\sigma_{*}(T).$ Moreover, it cannot be a limit ordinal.\\
(2)  Let $A$ be a compact complex subset.  The  Cantor-Bendixson rank of $x\in A$  is defined by  $\mbox{CBR}_{A}(x):=\mbox{max}\,\{\alpha: x \in \mbox{acc}^{\alpha}\,A\}.$ So if $T \in L(X)$ is  $g_{\sigma_{*}}^{\alpha}$-invertible for some ordinal $\alpha$ and $0\in \sigma_{*}(T),$ then $\mbox{CBR}_{\sigma_{*}(T)}(0)=\mbox{d}_{\sigma_{*}}(T)-1<\omega_{1}.$\\
(3) Clearly, if $\alpha$ is an ordinal, then $$\mbox{d}_{\sigma_{d}}(T)\leq\mbox{d}_{\sigma_{b}}(T)\leq\mbox{d}_{\sigma}(T)\leq\mbox{min}\,\{\mbox{CBR}(\sigma_{d}(T)), \mbox{d}_{\sigma_{d}}(T)+1\}$$ for every $g_{\sigma_{*}}^{\alpha}$-invertible operator $T,$ and if $\omega\leq\mbox{d}_{\sigma}(T),$ then $\mbox{d}_{\sigma}(T)=\mbox{d}_{\sigma_{b}}(T)=\mbox{d}_{\sigma_{d}}(T).$
\end{rema}

\begin{rema}\label{remarkinv} It follows from the proof of  Theorem \ref{thmpgsigmaetoile} that if $T$ is $g_{\sigma_{*}}^{\alpha}$-invertible for some successor  ordinal $\alpha>0,$ then there exists $(M,N)\in \mbox{Red}(T)$ such that $T_{M}$ is invertible, $\mbox{acc}^{\alpha-1}\,\sigma_{*}(T_{N})\subset\{0\},$  $\sigma(T_{M})\cap\sigma(T_{N})=\emptyset$ and $\sigma(T_{N})\setminus\{0\}\subset \Pi_{\sigma_{*}}^{\alpha}(T).$  So $\tilde{\sigma}:=\sigma(T_{N})$  is at most countable spectral set, $N=\R(P_{\tilde{\sigma}})$ and $M=\mathcal{N}(P_{\tilde{\sigma}}),$ where $P_{\tilde{\sigma}}$ is the spectral projection  of $T$  corresponding to $\tilde{\sigma}.$   This entails that the operator $T^{D_{\alpha}}_{\tilde{\sigma}}:=(T_{M})^{-1}\oplus 0_{N} \in \mbox{comm}^{2}(T).$ Clearly, $T^{D_{\alpha}}_{\tilde{\sigma}}$ is Drazin invertible,   $T^{D_{\alpha}}_{\tilde{\sigma}}TT^{D_{\alpha}}_{\tilde{\sigma}}=T^{D_{\alpha}}_{\tilde{\sigma}}$ and $\mbox{acc}^{\alpha-1}\,\sigma_{*}(T-T^{2}T^{D_{\alpha}}_{\tilde{\sigma}})\subset\{0\}.$  Note also that $\mbox{acc}^{\alpha-1}\,\sigma_{*}(TP_{\tilde{\sigma}})\subset\{0\}$ and  $T+P_{\tilde{\sigma}}=T_{M}\oplus (T+I)_{N}$  is  $g_{\sigma_{*}}^{\alpha-1}$-invertible.
\end{rema}
For a successor ordinal $\alpha<\omega_{1}$    and a    $g_{\sigma_{*}}^{\alpha}$-invertible operator  $T\in L(X),$ denote  by $$g_{\sigma_{*}}^{\alpha}\mathcal{D}(T):=\{(M,N)\in \mbox{Red}(T) \text{ such that } T_{M}  \text{ is invertible and  } \mbox{acc}^{\alpha-1}\,\sigma_{*}(T_{N})\subset\{0\}\}.$$
So $g_{\sigma_{*}}^{\alpha}\mathcal{D}(T)\subset g_{\sigma_{*}}^{\alpha+1}\mathcal{D}(T).$   The following proposition prove  that if in addition $T$ is not generalized Drazin invertible, then   $g_{\sigma_{*}}^{\alpha}\mathcal{D}(T)$ is at least  countable.
\begin{prop} Let $\alpha\geq 1$ be  a successor ordinal. If $T\in L(X)$ is $g_{\sigma_{*}}^{\alpha}$-invertible and  not generalized Drazin invertible, then there exists $\{(M_{m},N_{m})\}_{m}\subset g_{\sigma_{*}}^{n}\mathcal{D}(T)$  such that $(\sigma(T_{M_{m}}))_{m}$ is  a strictly increasing  sequence,  $(\sigma(T_{N_{m}}))_{m}$ is  a strictly decreasing   sequence    and $\sigma(T_{M_{m}})\cap\sigma(T_{N_{m}})=\emptyset$  for all $m\in \N.$  
\end{prop}
\begin{proof} Since   $T$ is $g_{\sigma_{*}}^{\alpha}$-invertible and  not generalized Drazin invertible,  there exists $\epsilon>0$ such that  $D(0,\epsilon)\cap [\mbox{acc}^{\alpha-1}\,\sigma_{*}(T)\setminus\{0\}]=\emptyset.$  Moreover, there exists   a positive strictly   decreasing sequence  $(\epsilon_{m})_{m}\subset\mathbb{R}$  such that  $\epsilon=\epsilon_{0},$ $C(0,\epsilon_{m})\cap\sigma(T)=\emptyset$    and       $D(0,\epsilon_{m+1})\cap\sigma(T)\subsetneq D(0,\epsilon_{m})  \cap\sigma(T).$ Denote by $\sigma_{m}:= D(0,\epsilon_{m})  \cap\sigma(T).$  It follows that  $(\sigma_{m})_{m}$ is a strictly decreasing    sequence of  countable spectral sets of  $\sigma(T).$   By taking $(M_{m},N_{m})=(\mathcal{N}(P_{\sigma_{m}}), \R(P_{\sigma_{m}})),$ where $P_{\sigma_{m}}$ is the spectral projection of $T$  corresponding to $\sigma_{m},$ we get the desired result.
\end{proof}

  Hereafter,  denote by $\mbox{Red}^{2}(T):=\{(M,N)\in \mbox{Red}(S): S\in \mbox{comm}(T)\}.$  
\begin{thm}\label{thmp3}  Let $\alpha\geq 1$ be a successor ordinal.   The  following statements are equivalent  for $T \in L(X)$:\\
(i) $T$  is $g_{\sigma_{*}}^{\alpha}$-invertible;\\
(ii) $0\notin \mbox{acc}^{\alpha}\,\sigma_{*}(T);$\\
(iii)  There exists $(M,N)\in \mbox{Red}^{2}(T)\cap g_{\sigma_{*}}^{\alpha}\mathcal{D}(T);$\\
(iv) There exists  a spectral set $\tilde{\sigma}$   of $T$ such that  $0\notin \sigma(T)\setminus\tilde{\sigma}$    and  $\tilde{\sigma}\setminus\{0\}\subset \Pi_{\sigma_{*}}^{\alpha}(T);$\\
(v) There exists a bounded projection $P\in \mbox{comm}(T)$    such that  $T+P$ is $g_{\sigma_{*}}^{\alpha-1}$-invertible    and  $\mbox{acc}^{\alpha-1}\,\sigma_{*}(TP)\subset\{0\};$\\
(vi) There exists a bounded projection $P\in \mbox{comm}^{2}(T)$    such that  $T+P$ is $g_{\sigma_{*}}^{\alpha-1}$-invertible     and  $\mbox{acc}^{\alpha-1}\,\sigma_{*}(TP)\subset\{0\};$\\
(vii)  There exists   $S \in \mbox{comm}(T)$  such that   $S^{2}T = S$ and  $\mbox{acc}^{\alpha-1}\,\sigma_{*}(T-T^{2}S)\subset\{0\};$\\  
(viii)  There exists   $S \in \mbox{comm}^{2}(T)$  such that   $S^{2}T = S$ and  $\mbox{acc}^{\alpha-1}\,\sigma_{*}(T-T^{2}S)\subset\{0\}.$
\end{thm}
\begin{proof} The equivalence   (i) $\Longleftrightarrow$ (ii) is proved in Theorem \ref{thmpgsigmaetoile}.  The implications   (i) $\Longrightarrow$ (iii),  (i) $\Longrightarrow$ (iv)  and (i) $\Longrightarrow$ (vi)  are proved in Remark \ref{remarkinv}. \\ 
(iv) $\Longrightarrow$ (i)  Let $\tilde{\sigma}$  be a spectral set   of $T$ such that  $0\notin \sigma(T)\setminus\tilde{\sigma}$    and  $\tilde{\sigma}\setminus\{0\}\subset \Pi_{\sigma_{*}}^{\alpha}(T).$  There exists   $(M,N)\in \mbox{Red}(T)$ such that $\sigma(T_{M})= \sigma(T)\setminus\tilde{\sigma}$ and $\sigma(T_{N})=\tilde{\sigma}.$ It is easily seen that  $(M,N)\in g_{\sigma_{*}}^{\alpha}\mathcal{D}(T),$ and then  $T$  is  $g_{\sigma_{*}}^{\alpha}$-invertible.\\
(iii) $\Longrightarrow$ (viii)  Let  $(M,N)\in \mbox{Red}^{2}(T)\cap g_{\sigma_{*}}^{\alpha}\mathcal{D}(T).$   The operator  $S=(T_{M})^{-1}\oplus 0_{N}$  is Drazin invertible. Moreover,   $S\in \mbox{comm}^{2}(T),$ $TS =I_{M}\oplus 0_{N},$ $S^{2}T= S$ and $T-T^{2}S =0_{M}\oplus T_{N}.$ Thus $\mbox{acc}^{\alpha-1}\,\sigma_{*}(T-T^{2}S)\subset\{0\}.$\\
(vii) $\Longrightarrow$  (i)  Suppose that there exists $S \in \mbox{comm}(T)$  such that $S^{2}T= S$ and  $\mbox{acc}^{\alpha-1}\,\sigma_{*}(T-T^{2}S)\subset\{0\}.$  Then    $TS$ is a projection and    $(M,N):=(\R(TS),\mathcal{N}(TS))\in \mbox{Red}(T)\cap  \mbox{Red}(S).$    It is not difficult to see that  $T_{M}$ is invertible and $S=(T_{M})^{-1}\oplus 0_{N}.$    Thus  $T-T^{2}S=0\oplus T_{N}.$ Consequently,   $T$  is  $g_{\sigma_{*}}^{\alpha}$-invertible.\\
(v) $\Longrightarrow$ (i) Suppose that there exists a bounded projection $P \in \mbox{comm}(T)$  such that   $T + P$ is $g_{\sigma_{*}}^{\alpha-1}$-invertible  and $\mbox{acc}^{\alpha-1}\,\sigma_{*}(TP)\subset\{0\}.$  Then $(A,B):=(\mathcal{N}(P),\R(P))\in \mbox{Red}(T).$    Therefore  $T_{A}$ is $g_{\sigma_{*}}^{\alpha-1}$-invertible    and $\mbox{acc}^{\alpha-1}\,\sigma_{*}(T_{B})\subset\mbox{acc}^{\alpha-1}\,\sigma_{*}(TP)\subset\{0\}.$  Hence    there exists $(C,D) \in Red(T_{A})$ such that $T_{C}$ is invertible and $\mbox{acc}^{\alpha-1}\,\sigma_{*}(T_{D})\subset\{0\}.$ Thus   $\mbox{acc}^{\alpha-1}\,\sigma_{*}(T_{D\oplus B})\subset\{0\},$  and then  $T$ is $g_{\sigma_{*}}^{\alpha}$-invertible.   This completes the proof of the theorem.
\end{proof}
   For $T \in L(X)$ and for a subspace $M$ of $X,$   $T_{X/M}$ means the operator induced by $T$ on the quotient space $X/M.$ The next proposition  is a consequence of Theorem \ref{thmpgsigmaetoile} and the fact that $\sigma_{*}(T)\subset\sigma_{*}(T_{M})\cup\sigma_{*}(T_{X/M}).$
\begin{prop}\label{proporth} Let $\alpha$  be an ordinal. Let $T\in L(X)$ and let  $M$ be  a closed $T$-invariant subspace. If  both $T_{M}$ and $T_{X/M}$ are   $g_{\sigma_{*}}^{\alpha}$-invertible, then $T$ is $g_{\sigma_{*}}^{\alpha}$-invertible.
\end{prop}
\begin{thm}\label{thmpsolwik..2} Let $\alpha\geq1$ be a successor ordinal. $T\in L(X)$ is $g_{\sigma_{*}}^{\alpha}$-invertible iff there exists a pair of  $T$-invariant subspaces  $(M,N)$ such that $N$ is closed, $X=M \oplus N,$   $T_{M}$ is invertible and $\mbox{acc}^{\alpha-1}\,\sigma_{*}(T_{N})\subset\{0\}.$  
\end{thm}
\begin{proof}  Let   $(M,N)$ be a pair  of  subspaces  such that $N$ is closed, $X=M \oplus N,$   $T_{M}$ is invertible and $\mbox{acc}^{\alpha-1}\,\sigma_{*}(T_{N})\subset\{0\}.$  Let $x\in X$ such that $T_{X/N}(x+N)=N.$  Then $Tx \in N,$ and  as  $T_{M}$ is one-to-one then  $x\in N.$  If  $x\in X,$    then   the fact that $T_{M}$ is onto implies that     $T(y)=P_{M}(x)$ for some $y\in M;$ where $P_{M}$ is the projection of $X$ onto $M.$   Thus   $T_{X/N}(y+N)=x+N$  for all $x\in X.$   Therefore $T_{X/N}$ is invertible.  So  there exists $\epsilon>0$ such that $B(0,\epsilon)\setminus\{0\}\subset(\sigma(T_{X/N}))^{C}\cap(\mbox{acc}^{\alpha-1}\,\sigma_{*}(T_{N}))^{C}\subset\mbox{acc}^{\alpha-1}\,\sigma_{*}(T).$  This entails, by  Theorem \ref{thmpgsigmaetoile}, that  $T$ is $g_{\sigma_{*}}^{\alpha}$-invertible. The converse is clear.
\end{proof}

\section{Almost  invertible operators}
\begin{dfn} Let $T\in L(X).$ We say that $T$ is  almost invertible  if there exists  $(M,N)\in \mbox{Red}(T)$   such that $T_{M}$  is invertible and $\sigma(T_{N})$ is  at most countable.   The set of   almost invertible operators  is  denoted by  $\mathcal{G}(X).$    
\end{dfn}
\begin{rema}\label{remasakold2} (i) Let $T\in L(X).$ Then  $\sigma(T)$ is at most countable iff $\sigma_{*}(T)$ is at most countable.\\
(ii)  We can easily see from  \cite[Theorem I.2.11]{mullerbook} that if  $T,S\in L(X)$  are commuting operators with $\sigma(T)$ and  $\sigma(S)$ are at most countable,  then $\sigma(T+S)$ and $\sigma(TS)$ are   at most countable.
\end{rema}
  Denote in the sequel by $\sigma_{al}(T)=\{\lambda \in \C : T- \lambda I \text{ is not almost invertible}\}.$ As a consequence of the previous remark  and the fact that $\mbox{acc}\,\sigma(T)\subset \sigma_{*}(T),$ we obtain  the following  result, which characterizes    almost invertible operators in terms  of   $\mbox{acc}^{\omega_{1}}\,\sigma_{*}(T).$ 
\begin{thm}\label{corcent}  Let   $T\in L(X)$ and let  $r_{*}:=\mbox{CBR}(\sigma_{*}(T)).$ Then  $T$ is almost invertible iff $T$ is  $g_{\sigma_{*}}^{\omega_{1}}$-invertible iff $T$ is  $g_{\sigma_{*}}^{r_{*}}$-invertible.  Consequently, $\displaystyle\mathcal{G}(X)=\mathcal{G}_{\sigma_{*}}^{\omega_{1}}(X)=\mathcal{G}^{r_{*}}_{\sigma_{*}}(X)$ and  $\displaystyle\sigma_{al}(T)=\mbox{acc}^{\omega_{1}}\,\sigma_{*}(T)=\mbox{acc}^{r_{*}}\,\sigma_{*}(T).$
\end{thm} 
\begin{proof} The proof is a consequence of the fact that a compact    complex  subset  $A$ is at most countable iff $\mbox{acc}^{r}\,A=\emptyset;$ where $r:=\mbox{CBR}(A)$  is a successor ordinal in the case of $A$ is not empty. 
\end{proof}
  From this proposition, it follows   that $\sigma_{al}(T)$  is compact and  $$\sigma_{*}(T)\setminus\sigma_{al}(T)=\bigsqcup_{\beta<\omega_{1}}\mbox{iso}\,(\mbox{acc}^{\beta}\,\sigma_{*}(T)).$$
For  an  almost invertible operator  $T\in L(X),$ denote in the sequel by: $$\mathcal{D}_{al}(T):=\{(M,N)\in \mbox{Red}(T) \text{ such that } T_{M}  \text{ is invertible and  } \sigma(T_{N})   \text{  is at most countable}\}.$$
Using  Theorems \ref{thmp3}, \ref{thmpsolwik..2}, \ref{corcent},   we obtain  the following theorem:   
\begin{thm}\label{thmp4}        The  following statements  hold  for every   $T \in L(X)$:\\
(i) $T$  is almost invertible;\\
(ii) $T$   is  $g_{\sigma_{*}}^{\omega_{1}}$-invertible;\\
(iii) $0\notin \mbox{acc}^{\omega_{1}}\,\sigma_{*}(T);$\\
(iv)  There exists $(M,N)\in \mbox{Red}^{2}(T)\cap \mathcal{D}_{al}(T);$\\
(v) There exists a pair of  $T$-invariant subspaces $(M,N)$ such that $N$ is closed, $X=M \oplus N,$   $T_{M}$ is invertible and $\sigma_{*}(T_{N})$ is at most countable;\\  
(vi) There exists  an at most countable spectral set $\tilde{\sigma}$   of $T$ such that  $0\notin \sigma(T)\setminus\tilde{\sigma};$\\
(vii) There exists  a spectral set $\tilde{\sigma}$   of $T$ such that  $0\notin \sigma(T)\setminus\tilde{\sigma}$    and  $\tilde{\sigma}\setminus\{0\}\subset \displaystyle\bigsqcup_{\beta<\omega_{1}}\mbox{iso}\,(\mbox{acc}^{\beta}\,\sigma(T));$\\
(viii) There exists a bounded projection $P\in \mbox{comm}(T)$    such that  $T+P$ is almost invertible    and  $\sigma_{*}(TP)$ is at most countable;\\
(ix) There exists a bounded projection $P\in \mbox{comm}^{2}(T)$    such that  $T+P$ is almost invertible       and  $\sigma_{*}(TP)$  is at most countable;\\
(x)  There exists   $S \in \mbox{comm}(T)$ such that   $S^{2}T = S$ and  $\sigma_{*}(T-T^{2}S)$  is at most countable;\\  
(xi)  There exists   $S \in \mbox{comm}^{2}(T)$ such that   $S^{2}T = S$ and  $\sigma_{*}(T-T^{2}S)$  is at most countable.
\end{thm}
Recall that an operator $R \in L(X)$ is called Riesz if $\sigma_{b}(R)\subset\{0\}.$
\begin{cor}   Let $T \in L(X).$   The  following statements  hold:\\
(i)  $T$  is almost invertible iff   $T^{*}$  is almost invertible.\\
(ii) If $S$ is a bounded operator acts on a complex Banach space $Y,$ then  $T\oplus S$   is almost invertible iff both $T$ and $S$   are    almost invertible. \\
 (iii)   $T$  is  almost invertible  iff   $T^{m}$  is  almost invertible  for some  (equivalently for every) integer $m\geq 1.$\\
(iv)  Let   $M$ be  a closed $T$-invariant subspace. If both   $T_{M}$ and $T_{X/M}$ are  almost invertible, then $T$ is almost invertible.\\
(v) If  $R\in \mbox{comm}(T)$ is a Riesz operator,  then $\sigma_{al}(T+R)=\sigma_{al}(T).$
\end{cor}
\begin{prop} Let  $T,R\in L(X)$ be      almost  invertible operators such that $TR=RT=0.$ Then $T+R$ is   almost invertible.
\end{prop}
\begin{proof}  Let $S\in \mbox{comm}^{2}(T)$ and $S^{'}\in \mbox{comm}^{2}(R)$ such that $S^{2}T = S,$ $(S^{'})^{2}R = S^{'},$      $\sigma_{*}(T-T^{2}S)$ and  $\sigma_{*}(R-R^{2}S^{'})$  are  at most countable. It is easy to see that  $(S+S^{'})\in \mbox{comm}^{2}(T),$ $(S+S^{'})^{2}(T+R)=S^{2}T+(S^{'})^{2}T=S+S^{'}$ and $(T+R)-(T+R)^{2}(S+S^{'})=T-T^{2}S+R-R^{2}S^{'}.$  We conclude from Remark \ref{remasakold2} and Theorem \ref {thmp4} that  $T+R$ is almost countable.
\end{proof}
According to \cite{Grabiner},  an operator  $T$ is said to have a   topological uniform descent  if  $\R(T) + \mathcal{N}(T^{d})$  is closed for some integer $d$ and  $\R(T) + \mathcal{N}(T^{n}) = \R(T) + \mathcal{N}(T^{n})$ for all $n \geq d.$   For the definition of the SVEP, one can   see \cite{Jiang-zhong}.
\begin{prop}  Let $T \in L(X).$  The following assertions hold:\\
(i) If $0\notin \mbox{acc}^{\omega_{1}}\,\sigma(T),$ then  $T$ and  its dual adjoint  $T^{*}$  have the SVEP at $0.$ \\
(ii) If  $T$ has a topological uniform descent, then $T$ is Drazin invertible iff $T$ is  almost invertible. 
\end{prop}
\begin{proof}  The point (i) is a consequence of   Theorem  \ref{corcent}, and  the point (ii)   is a consequence of   \cite[Theorems 3.2, 3.4]{Jiang-zhong}.
\end{proof}
In the next corollary, we denote, respectively,  the essential spectrum   and the B-Fredholm  spectrum of $T\in L(X)$  by   $\sigma_{e}(T)$ and  $\sigma_{bf}(T).$   For the  definition  and the properties of the B-Fredholm spectrum, one can see \cite{aznay-ouahab-zariouh,berkaniBF}.
\begin{cor}    $\sigma_{b}(T)=\sigma_{e}(T)\cup \mbox{acc}^{\omega_{1}}\,\sigma_{*}(T)$
and   $\sigma_{d}(T)=\sigma_{bf}(T)\cup\mbox{acc}^{\omega_{1}}\,\sigma_{*}(T),$ for all   $T \in L(X).$ 
\end{cor}

\goodbreak 
{\small  \noindent  Zakariae Aznay,\\  Laboratory (L.A.N.O), Department of Mathematics,\\Faculty of Science, Mohammed I University,\\  Oujda 60000 Morocco.\\
aznay.zakariae@ump.ac.ma\\

 \noindent Abdelmalek Ouahab,\newline Laboratory (L.A.N.O), Department of
	Mathematics,\newline Faculty of Science, Mohammed I University,\\
	  Oujda 60000 Morocco.\\
	  ouahab05@yahoo.fr\\

\noindent   Hassan  Zariouh,\newline Department of
Mathematics (CRMEFO),\newline
   and laboratory (L.A.N.O), Faculty of Science,\newline
  Mohammed I University, Oujda 60000 Morocco.\\
   h.zariouh@yahoo.fr

\end{document}